\documentclass{amsart}

\usepackage{amsmath,amssymb,latexsym,amsthm,newlfont,enumerate}

\usepackage{graphicx}
\usepackage[all]{xy}

\theoremstyle{plain}

\newtheorem{thm}{Theorem}[section]

\newtheorem{pro}[thm]{Proposition}

\newtheorem{lem}[thm]{Lemma}

\newtheorem{cor}[thm]{Corollary}

\newtheorem{assumption}[thm]{Assumption} 

\theoremstyle{definition}

\theoremstyle{remark}

\newcommand{\Z}{\mathbb{Z}}

\newcommand{\C}{\mathbb{C}}
\newcommand{\R}{\mathbb{R}}
\newcommand{\Q}{\mathbb{Q}}

\newcommand{\OO}{\mathcal{O}}

\newcommand{\mcal}{\mathcal}
\newcommand{\id}{\mathrm{id}}

\DeclareMathOperator{\Sing}{Sing}

\DeclareMathOperator{\GL}{GL}
\DeclareMathOperator{\SL}{SL}

\DeclareMathOperator{\Aut}{Aut}
\DeclareMathOperator{\Bir}{Bir}

\DeclareMathOperator{\Nef}{Nef}

\begin{document}
\title{Automorphisms of Calabi-Yau threefolds with Picard number three}

\author{Vladimir Lazi\'c}
\address{Mathematisches Institut, Universit\"at Bonn, Endenicher Allee 60, 53115 Bonn, Germany}
\email{lazic@math.uni-bonn.de}

\author{Keiji Oguiso}
\address{Department of Mathematics, Osaka University, Toyonaka 560-0043 Osaka,
Japan and Korea Institute for Advanced Study, Hoegiro 87, Seoul, 130-722, Korea}
\email{oguiso@math.sci.osaka-u.ac.jp}

\author{Thomas Peternell}
\address{Mathematisches Institut, Universit\"at Bayreuth, 95440 Bayreuth, Germany}
\email{thomas.peternell@uni-bayreuth.de}

\dedicatory{Dedicated to  Professor Yujiro Kawamata on the occasion of his 60th birthday}

\thanks{All authors were partially supported by the DFG-Forschergruppe 790 ``Classification of Algebraic Surfaces and Compact Complex Manifolds". The first author was partially supported by the DFG-Emmy-Noether-Nachwuchsgruppe ``Gute Strukturen in der h\"oherdimensionalen birationalen Geometrie". The second author is supported by JSPS Grant-in-Aid (S) No 25220701, JSPS Grant-in-Aid (S) No 22224001, JSPS Grant-in-Aid (B) No 22340009, and by KIAS Scholar Program.}

\begin{abstract}
We prove that the automorphism group of a Calabi-Yau threefold with Picard number three is either finite, or isomorphic to the infinite cyclic group up to finite kernel and cokernel.
\end{abstract}

\maketitle

\section{Introduction}

In this paper we are interested in the automorphism group of a Calabi-Yau threefold with small Picard number. Here, a Calabi-Yau threefold is a smooth complex projective threefold $X$ with trivial canonical bundle $K_X$ such that $h^1(X,\OO_X)=0$. 

It is a classical fact that the group of birational automorphisms $\Bir(X)$ and the automorphism group $\Aut(X)$ are finite groups and coincide when $X$ is a Calabi-Yau threefold with $\rho(X)=1$. It is, however, unknown which finite groups really occur as automorphism groups, even for smooth quintic threefolds. When $\rho(X)=2$, the automorphism group is also finite by \cite[Theorem 1.2]{Og12} (see also \cite{LP12}), while there is an example of a Calabi-Yau threefold with $\rho(X)=2$ and with infinite $\Bir(X)$ \cite[Proposition 1.4]{Og12}.

In contrast, Borcea \cite{Bor91b} gave an example of a Calabi-Yau threefold with $\rho(X)=4$ having infinite automorphism group, and the same phenomenon is expected for any Picard number $\rho(X)\geq4$; for examples with large Picard numbers, see \cite{GM93,OT13}. 

Thusfar, the case of Picard number $3$ remained unexplored. Perhaps surprisingly, we show that the automorphism groups of such threefolds are relatively small: 

\begin{thm} \label{thm:main}
Let $X$ be a Calabi-Yau threefold with $\rho(X) = 3$. 

Then the automorphism group $\Aut(X)$ is either finite, or it is an almost abelian group of rank $1$, i.e.\ it is isomorphic to $\Z$ up to finite kernel and cokernel.
\end{thm}

We investigate automorphisms $g$ of infinite order and distinguish the cases when $g$ has an eigenvalue different than $1$, and when $g$ only has eigenvalue $1$.  Theorem \ref{thm:main}  then follows from Corollary \ref{cor:1} and Proposition \ref{pro:cubic2} below. 

\vskip 2mm
At the moment, we do not have an example where $\Aut(X)$ is an infinite group. Existence of such an example would show that $3$ is the smallest possible Picard number of a Calabi-Yau threefold with infinite automorphism group. However, finiteness of the automorphism group is known when the fundamental group of $X$ is infinite: when $X$ is a Calabi-Yau threefold of Type A, i.e.\ $X$ is an \'etale quotient of a torus, then $\Aut(X)$ is finite by \cite[Theorem (0.1)(IV)]{OS01}.  The case when $X$ is of Type K, i.e.\ $X$ is an \'etale quotient of a product of an elliptic curve and a K3 surface, of Picard number $\rho(X) \leq 3$, is studied in the forthcoming work \cite{HK13}. 

\vskip 2mm
It is our honour to dedicate this paper to Professor Yujiro Kawamata on the occasion of his sixtieth birthday. This article and our previous papers \cite{Og12,LP12} are inspired by his beautiful paper \cite{Kaw97}.
\section{Preliminaries}

We first fix some notation. Let $X$ be a Calabi-Yau threefold with Picard number $\rho(X) = 3$. The automorphism group of $X$ is denoted by ${\Aut}(X)$ and $N^1(X)$ is the N\'eron-Severi group of $X$ generated by the numerical classes of line bundles on $X$. Note that $N^1(X)$ is a free $\mathbb Z$-module of rank $3$. There is a natural homomorphism
$$ r\colon {\Aut}(X) \to {\GL}(N^1(X)),$$ 
and we set $\mathcal A(X) = r({\Aut}(X))$. Note that the kernel of $r$ is finite \cite[Proposition 2.4]{Og12}, hence ${\Aut}(X)$ is finite if and only if $\mathcal A(X)$ is finite. We furthermore let $N^1(X)_{\R} := N^1(X) \otimes \mathbb R$
be the vector space generated by $N^1(X)$. 

\begin{pro}\label{pro:plane}
Let $\ell_1$ and $\ell_2$ be two distinct lines in $\R^2$ through the origin, and let $G$ be a subgroup of $\GL(2,\Z)$ which acts on $\ell_1\cup\ell_2$. 

If $G$ is infinite, then it is an almost abelian group of rank $1$, i.e. $G$ contains an abelian subgroup of finite index. 
\end{pro}
\begin{proof}
The proof follows from that of \cite[Theorem 3.9]{LP12}, and we recall the argument for the convenience of the reader. Fix nonzero points $x_1\in\ell_1$ and $x_2\in\ell_2$. Then for any $g\in G$ there exist a permutation $(i_1,i_2)$ of the set $\{1,2\}$ and real numbers $\alpha_1$ and $\alpha_2$ such that $gx_1=\alpha_1 x_{i_1}$ and $gx_2=\alpha_2 x_{i_2}$. It follows that there are positive numbers $\beta_1$ and $\beta_2$ such that $g^4x_i=\beta_i x_i$. Hence, taking a quotient of $G$ by a finite group, we may assume that $G$ acts on $\R_+x_1$ and $\R_+x_2$.

For every $g\in G$, let $\alpha_g$ be the positive number such that $gx_1=\alpha_g x_1$, and set $\mcal S=\{\alpha_g\mid g\in G\}$. Then $\mcal S$ is a multiplicative subgroup of $\R^*$ and the map 
$$ G \to \mcal S, \quad g \mapsto \alpha_g$$
is an isomorphism of groups. It  therefore suffices to show that $\mcal S$ is an infinite cyclic group. By \cite[21.1]{Fo81}, it is enough to prove that $\mathcal S$ is discrete. Otherwise, we can pick a sequence $(g_i)$ in $G$ such that $(\alpha_{g_i})$ converges to $1$. Fix two linearly independent points $h_1,h_2\in\Z^2$. Then $g_ih_1\to h_1$ and $g_ih_2\to h_2$ when $i\to\infty$. Since $g_ih_1,g_ih_2\in\Z^2$, this implies that $g_ih_1=h_1$ and $g_ih_2=h_2$ for $i\gg0$, and hence $g_i=\id$ for $i\gg0$. 
\end{proof}

In order to prove our main result, Theorem \ref{thm:main}, we first show that the cubic form on our Calabi-Yau threefold $X$ always splits in a special way, and this almost immediately has strong consequences on the structure of the automorphism group.

In this paper, when $L$ is a linear, quadratic or cubic form on $N^1(X)_\R$, we do not distinguish between $L$ and the corresponding locus $(L=0)\subseteq N^1(X)_\R$.

We start with the following lemma.

\begin{lem} \label{inf}
Let $X$ be a Calabi-Yau threefold with Picard number $3$. Assume that $\Aut(X)$ is infinite. 

Then there exists $g\in\mathcal A(X)$ with $\det g=1$ such that $\langle g\rangle\simeq\Z$.
\end{lem}
\begin{proof}
By  possibly replacing $\mcal A(X)$ by  the subgroup $\mcal A(X) \cap SL(N^1(X))$ of index at most $2$, we may assume that all elements of $\mcal A(X)$ have determinant $1$. Assume that all elements of $\mcal A(X)$ have finite order, and fix an element 
$$h\in\mcal A(X)\subseteq\GL(N^1(X))$$ 
of order $n_h$. Since $\rho(X)=3$, the characteristic polynomial $\Phi_h(t)\in\Z[t]$ of $h$ is of degree $3$. If $\xi$ is an eigenvalue of $h$, then $\xi^{n_h}=1$, and hence $\varphi(n_h)\leq3$, where $\varphi$ is Euler's function. An easy calculation shows that then $n_h\leq6$, and therefore $\mcal A(X)$ is a finite group by Burnside's theorem, a contradiction. 
\end{proof}

If $c_2(X) = 0$ in $H^4(X,\mathbb R)$, then ${\Aut}(X)$ is finite by \cite[Theorem (0.1)(IV)]{OS01}. Combining this with Lemma \ref{inf}, we may assume the following:

\begin{assumption} \label{assumption}
Let $X$ be a Calabi-Yau threefold with Picard number $3$. We assume that $c_2(X) \ne 0$ and that $\Aut(X)$ is infinite, and we fix an element $g \in \mathcal A(X)$ of infinite order as given in Lemma \ref{inf}. We denote by $C$ the cubic form on $N^1(X)_\R$ given by the intersection product.
\end{assumption} 

\begin{pro}\label{pro:uvw}
Let $h \in \mcal A(X)$.
\begin{enumerate}
\item[(i)] If $h$ is of infinite order, then there exist a real number $\alpha \geq 1$ and (when $\alpha=1$ not necessarily distinct) nonzero elements $u,v,w\in N^1(X)_{\R}$ such that $w$ is integral, $v$ is nef, and 
$$hu=\frac1\alpha u,\quad hv=\alpha v,\quad hw=w.$$ 
Moreover, if $\alpha =1$, then $\alpha$ is the unique eigenvalue of (the complexified) $h$.
\item[(ii)] If $h \neq \id$ has finite order, then (the complexified) $h$ has eigenvalues $1$, $\lambda$, $\bar\lambda$, where $\lambda\in\{\pm i, \pm (\frac12 \pm i\frac{\sqrt 3}2)\}$. 
\end{enumerate}
\end{pro}
\begin{proof}
Let $h^*$ denote the dual action of $h$ on $H^4(X,\Z)$. Since $h^*$ preserves the second Chern class $c_2(X)\in H^4(X,\Z)$, one of its eigenvalues is $1$, 
and therefore $h$ also has an eigenvector $w$ with eigenvalue $1$. Since $h$ acts on the nef cone $\Nef(X)$, by the Birkhoff-Frobenius-Perron theorem \cite{Bir67} there exist $\alpha\geq1$ and 
$v\in\Nef(X)\setminus\{0\}$ such that $hv=\alpha v$. As $\det h=1$, if $\alpha>1$, then the remaining eigenvalue of $h$ is $1/\alpha$. 

Assume that $\alpha=1$. Then by the Birkhoff-Frobenius-Perron theorem, all eigenvalues of $h$ have absolute value 1. Thus the characteristic polynomial of $h$ reads 
$$ \Phi_h(t) = (t-1)(t-\lambda)(t-\bar\lambda)$$
with $\vert \lambda \vert = 1$. Since $\Phi_h$ has integer coefficients, a direct calculation gives $\lambda\in\{1,\pm i,\pm (\frac12 \pm i\frac{\sqrt 3}2)\}$. When $\lambda\neq1$, it is easily checked that $h$ has finite order. 

Finally, if $\lambda=1$, then the Jordan form of $h$ is 
$$\text{either}\quad
\left(
\begin{array}{ccc}
1 & 1 & 0 \\
0 & 1 & 1 \\
0 & 0 & 1
\end{array}
\right)
\quad\text{or}\quad
\left(
\begin{array}{ccc}
1 & 1 & 0 \\
0 & 1 & 0 \\
0 & 0 & 1
\end{array}
\right).$$
In both cases it is clear that $h$ has infinite order.
\end{proof}

In the following two sections, we fix an element of infinite order as in Lemma \ref{inf} and analyse separately the cases $\alpha>1$ and $\alpha=1$ as in Proposition \ref{pro:uvw}(i). 
\section{The case $\alpha>1$}

\begin{pro}\label{pro:uvwBigger}
Under Assumption \ref{assumption} and in the notation from Proposition \ref{pro:uvw} for $h=g$, assume that $\alpha>1$. Then $u$ and $v$ are nef and irrational, we have 
\begin{equation}\label{eq:relations}
u^3=v^3=u^2v=uv^2=u^2w=uw^2=v^2w=vw^2=0,
\end{equation} 
and the plane $\R u+\R v$ is in the kernel of the linear form given by $c_2(X)\in H^4(X,\Z)$.
\end{pro}
\begin{proof}
We first need to show that the eigenspace of $1/\alpha$ intersects $\Nef(X)$. Pick $u\neq0$ such that $gu=\frac1\alpha u$, 
and note that $u,v$ and $w$ form a basis of $N^1(X)_\R$. Take a general ample class 
$$H=xv+yu+zw,$$ 
and observe that $y\neq0$ by the general choice of $H.$ Then $g^{-n}H$ is ample for every positive integer $n$, hence the divisor
$$\lim_{n\to\infty}\frac1{\alpha^n|y|}g^{-n}H=\lim_{n\to\infty}\Big(\frac x{\alpha^{2n}|y|}v+\frac y{|y|}u+\frac z{\alpha^n |y|}w\Big)=\frac y{|y|}u$$
is nef. Now replace $u$ by $yu/|y|$ if necessary to achieve the nefness of $u.$ 

Furthermore, since $v^3=(gv)^3=\alpha^3 v^3$, we obtain $v^3=0$; other relations in 
\eqref{eq:relations} are proved similarly. Also, 
$$v\cdot c_2(X)=gv\cdot gc_2(X)=\alpha v\cdot c_2(X),$$
hence $v\cdot c_2(X)=0$, and analogously $u\cdot c_2(X)=0$. 

Finally, assume that $v$ is rational. By replacing $v$ by a rational multiple, we may assume that $v$ is a primitive element of $N^1(X)$. But the eigenspace associated to $\alpha$ is $1$-dimensional, and since $gv$ is also primitive, we must have $gv=v$, a contradiction. Irrationality of $u$ is proved in the same way. 
\end{proof}

\begin{pro}\label{pro:cubic}
Under Assumption \ref{assumption} and in the notation of Proposition \ref{pro:uvw} for $h=g$, assume that $\alpha>1$. Let $L$ be the linear form on $N^1(X)_\R$ given by $c_2(X)$. 

Then one of the following holds:
\begin{enumerate}
\item[(i)] $C=L_1L_2L$, where $L_1$ and $L_2$ are irrational linear forms such that 
$$L_1\cap L_2=\R w,\quad L_1\cap L=\R u,\quad L_2\cap L=\R v;$$
\begin{figure}[htb]
\begin{center}
\includegraphics[width=0.44\textwidth]{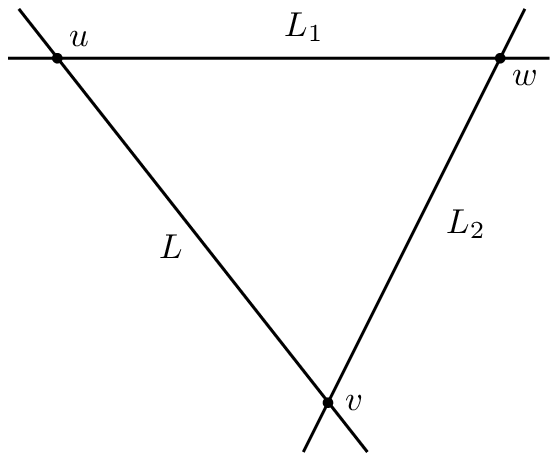}
\end{center}
\end{figure}
\item[(ii)] $C=QL$, where $Q$ is an irreducible quadratic form. Then 
$$Q\cap L=\R u\cup\R v,$$ 
and the planes $\R u+\R w$ and $\R v+\R w$ are tangent to $Q$ at $u$ and $v$ respectively.
\begin{figure}[htb]
\begin{center}
\includegraphics[width=0.44\textwidth]{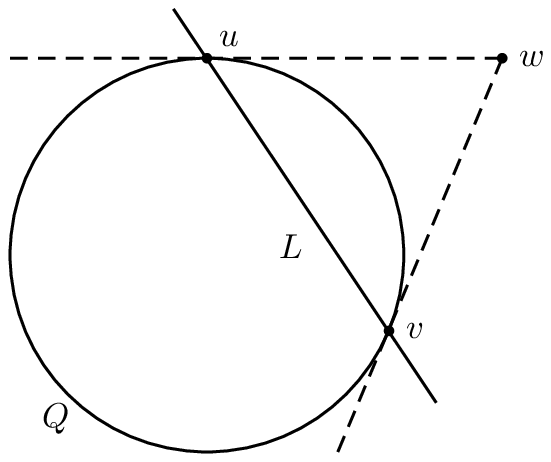}
\end{center}
\end{figure}
\end{enumerate}
\end{pro}
\begin{proof}
Denote $A=w^3$ and $B=uvw$. We first claim that $B\neq0$. In fact, suppose that $B = 0$ and let $H$ be any ample class. Then the relations \eqref{eq:relations} imply $uv=0$, hence $0=(Huv)^2=(H^2u)\cdot(v^2u)$, 
and the Hodge index theorem \cite[Corollary 2.5.4]{BS95} yields that $H$ and $v$ are proportional, which is a contradiction since $v^3=0$. This proves the claim.

Therefore, for any real variables $x,y,z$ we have
$$(xu+yv+zw)^3=z(Az^2+6Bxy),$$
and thus in the basis $(u,v,w)$ we have $C=QL$, where $Q=Az^2+6Bxy$. We consider two cases.

Assume first that $A=0$. Then $C=6Bxyz$, and we set $L_2=x$ and $L_1=6By$. This gives (i).

If $A\neq0$, then 
$$Q=Az^2+6Bxy=
\left(
\begin{array}{c}
x \\
y \\
z
\end{array}
\right)^t
\left(
\begin{array}{ccc}
0 & 3B & 0 \\
3B & 0 & 0 \\
0 & 0 & A
\end{array}
\right)
\left(
\begin{array}{c}
x \\
y \\
z
\end{array}
\right),$$
and the signature of $Q$ is $(2,1)$. Therefore, $Q$ is a non-empty smooth quadric. It is now easy to see that the tangent plane to $Q$ at $u$ is $(y=0)$, and the tangent plane to $Q$ at $v$ is $(x=0)$. This proves the proposition. 
\end{proof}

\begin{cor}\label{cor:1}
Under Assumption \ref{assumption} and in the notation of Proposition \ref{pro:uvw} for $h=g$, assume that $\alpha>1$. Then $\mcal A(X)$ is an almost abelian group of rank $1$.
\end{cor}
\begin{proof}
First note that every element $h\in\mcal A(X)$ fixes the cubic $C$ and the plane $L=c_2(X)^\perp$. Further, the singular locus $\Sing(C)$ of $C$ is $h$-invariant. In the case (i) of Proposition \ref{pro:cubic}, $\Sing(C)=\R u\cup\R v\cup\R w$. This implies that the set 
$$\R u\cup \R v\subseteq L$$
is $h$-invariant, and hence so is $\R w$. In particular, the sets $\R u$, $\R v$, $\R w$ are each $h^2$-invariant. Then Proposition \ref{pro:uvw} immediately shows that $hw=w$, and hence the map 
$$ \mcal A(X) \to \GL(2,\mathbb Z), \quad  h \mapsto h|_L$$
is injective. Now the claim follows from Proposition \ref{pro:plane}. 

In the case (ii) of Proposition \ref{pro:cubic}, we have $ \Sing(C) = \R u\cup\R v\subseteq L$, and $\R w$ is $h$-invariant as it is the intersection of tangent planes to $Q$ at $u$ and $v$. Now we conclude similarly as above. 
\end{proof}
\section{The case $\alpha=1$}

\begin{lem}\label{lem:jordanform}
Under Assumption \ref{assumption} and in the notation of Proposition \ref{pro:uvw} for $h=g$, assume that $\alpha=1$. Then the Jordan form of $g$ is
\begin{equation}\label{eq:jordan}
\left(
\begin{array}{ccc}
1 & 1 & 0 \\
0 & 1 & 1 \\
0 & 0 & 1
\end{array}
\right).
\end{equation}  
In particular, the eigenspace of $g$ associated to the eigenvalue $1$ has dimension $1$.
\end{lem}
\begin{proof} By Proposition \ref{pro:uvw}, $\alpha = 1$ is the unique eigenvalue of $g$. 
Therefore the Jordan form of $g$ is either of the form  \eqref{eq:jordan} or of the form
\begin{equation}\label{eq:matrix}
\left(
\begin{array}{ccc}
1 & 1 & 0 \\
0 & 1 & 0 \\
0 & 0 & 1
\end{array}
\right).
\end{equation}  
Assume that the Jordan form is of the form \eqref{eq:matrix}; in other words, there is a basis $(u_1,u_2,u_3)$ of $N^1(X)_\R$ such that 
$$gu_1=u_1,\quad gu_2=u_1+u_2,\quad gu_3=u_3.$$
Clearly,
$$g^nu_2=u_2+nu_1$$
for every integer $n$, and furthermore,
$$u_2^3=(g^nu_2)^3=u_2^3+3nu_2^2u_1+3n^2u_2u_1^2+n^3u_1^3.$$
This gives
\begin{equation}\label{eq:1}
u_2^2u_1=u_2u_1^2=u_1^3=0.
\end{equation}
Similarly, from the equations 
$$u_2^2u_3=(g^nu_2)^2g^nu_3\quad\text{and}\quad u_2u_3^2=g^nu_2(g^nu_3^2)^2$$ 
we get
\begin{equation}\label{eq:2}
u_1^2u_3=u_1u_3^2=u_1u_2u_3=0.
\end{equation}
For any smooth very ample divisor $H$ on $X$, \eqref{eq:1} and \eqref{eq:2} give $u_1^2 \cdot  H = u_1 \cdot H^2 = 0,$ thus 
$(u_1|_H)^2=0$ and $u_1|_H\cdot H|_H=0$, and hence $u_1|_H=0$, applying the Hodge index theorem on $H.$ 
This implies $u_1=0$ by the Lefschetz hyperplane section theorem, a contradiction. Thus  the Jordan form cannot be of type \eqref{eq:matrix}, and the assertion is proved. 
\end{proof}

\begin{pro}\label{pro:jordan2}
Under Assumption \ref{assumption} and in the notation of Proposition \ref{pro:uvw} for $h=g$, assume that $\alpha=1$. Then, possibly by rescaling $w$, there exist $w_1,w_2\in N^1(X)$ such that $(w,w_1,w_2)$ is a basis of $N^1(X)_\R$ with respect to the Jordan form \eqref{eq:jordan}, and we have
\begin{equation}\label{eq:3}
w\cdot c_2(X)=w_1\cdot c_2(X)=w^2=w_1^3=ww_1^2=ww_1w_2=0
\end{equation}
and
\begin{equation}\label{eq:4}
ww_2^2=2w_1w_2^2=-2w_1^2w_2 \ne 0.
\end{equation}
\end{pro}
\begin{proof}
Pick any $w_2\in N^1(X)$ such that $w_1 :=(g-\id)w_2\neq0$ and $u:=(g-\id)^2w_2\neq0$, which is possible by Lemma \ref{lem:jordanform}.
Then
$$gu=u,\quad gw_1=u+w_1,\quad gw_2=w_1+w_2,$$
and it is easy to check that $(u,w_1,w_2)$ is a basis of $N^1(X)_\R$. Since the eigenspace associated to the eigenvalue $1$ of $g$ is $1$-dimensional by Lemma \ref{lem:jordanform}, by Proposition \ref{pro:uvw} we may assume that $u=w$. We first observe that 
$$g^nw_1=w_1+nw\quad\text{and}\quad g^nw_2=w_2+nw_1+\frac{n(n-1)}{2}w$$
for any integer $n$. Then the equations 
$$w_1\cdot c_2(X)=g^nw_1\cdot c_2(X)\quad\text{and}\quad w_2\cdot c_2(X)=g^nw_2\cdot c_2(X)$$ 
give 
$$w\cdot c_2(X)=w_1\cdot c_2(X)=0.$$
Similarly, from $w_1^3=(g^nw_1)^3$ and $ww_2^2=(g^nw)(g^nw_2)^2$ we get 
$$ww_1^2=ww_1w_2= w^2=0,$$ 
and $w_1^2w_2=(g^nw_1)^2(g^nw_2)$ yields 
$$w_1^3=0.$$ 
Finally, from $w_2^3=(g^nw_2)^3$ we obtain \eqref{eq:4}, up to the non-vanishing statement. Assume that $ww_2^2 = 0$. Since $w,w_1,w_2$ generate $N^1(X)_\R$, this implies that for any two smooth very ample line bundles $H_1$ and $H_2$ on $X$ we  have $w\cdot H_1\cdot H_2 = 0$, and in particular $w|_{H_1}=0$. But then $w=0$ by the Lefschetz hyperplane section theorem, a contradiction.
\end{proof}

\begin{pro}\label{pro:cubic2}
Under Assumption \ref{assumption} and in the notation of Proposition \ref{pro:uvw} and Proposition \ref{pro:jordan2} for $h=g$, assume that $\alpha=1$.
\begin{enumerate}
\item[(i)] Let $L$ be the linear form on $N^1(X)_\R$ given by $c_2(X)$. Then $C=QL$, where $Q$ is an irreducible quadratic form, and $L$ is tangent to $Q$ at $w$.
\begin{figure}[htb]
\begin{center}
\includegraphics[width=0.365\textwidth]{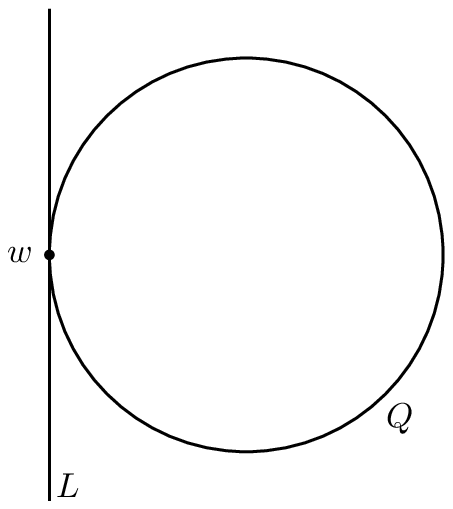}
\end{center}
\end{figure}
\item[(ii)] The automorphism group $\Aut(X)$ is an almost abelian group of rank $1$.
\end{enumerate}
\end{pro}
\begin{proof}
Set  $E=3ww_2^2/2$ and $F=w_2^3$. Then, using \eqref{eq:3} and \eqref{eq:4}, for all real variables $x,y,z$ we obtain the equation
$$(xw+yw_1+zw_2)^3=z(Fz^2+2Exz-Ey^2+Eyz).$$ 
Since $L=z$ by \eqref{eq:3}, we have $C=QL$, where $Q=Fz^2+2Exz-Ey^2+Eyz$. Noticing that $E \ne 0$ by Proposition \ref{pro:jordan2}, the tangent plane to $Q$ at $w$ is $(z=0)$. This shows (i).
\vskip 2mm
For (ii), consider any $h \in \mcal A(X)$. We may assume $\det h=1$, possibly replacing $\mcal A(X)$ by $\mcal A(X) \cap \SL(N^1(X)).$ 

The singular locus of $C$ is $\R w$, hence $\R w$ is $h$-invariant and therefore defined over $\mathbb Q$. By the shape of the cubic and by Proposition \ref{pro:cubic}, and since the element $g$ in Assumption \ref{assumption} is chosen arbitrarily, $h$ has a unique real eigenvalue $\alpha=1$. By Proposition \ref{pro:uvw} and by Lemma \ref{lem:jordanform}, $\R w$ is the only eigenspace of $h$, thus $hw = w$. 

The plane $L = c_2(X)^{\perp}$ is $h$-invariant, and note that $L$ is spanned by $w$ and $w_1$ by \eqref{eq:3}. In the basis $(w,w_1)$, the restriction $h|_L$ has the form 
$$
\left(
\begin{array}{ccc}
1 & a_h\\
0 & b_h
\end{array}
\right),
$$
and $\det (h|_L) = \pm 1$. 
By possibly replacing $\mcal A(X)$ by the preimage of $ \mcal A(X) \vert_L \cap \SL(L)$
under the restriction map $ \mcal A(X) \to \mcal A(X) \vert_L$, which has index at most $2$, we may assume that $\det (h|_L) = 1$, and thus $b_h = 1$. Hence, the matrix of $h$ in the basis $(w,w_1,w_2)$ is 
\begin{equation}\label{eq:jA}
\mathcal H=\left(
\begin{array}{ccc}
1 & a_h & d_h \\
0 & 1 & c_h \\
0 & 0 & 1
\end{array}
\right).
\end{equation}  
This implies, in particular, that $h$ cannot be of finite order. The quadric $Q$ is given in this basis by the matrix 
$$ \mathcal Q = \left(
\begin{array}{ccc}
0 & 0 & E \\
0 & -E &  \frac12 E \\
E & \frac12 E & F
\end{array}
\right).$$
We now view $Q$ as a quadric over $\C$. Since $Q$ is $h$-invariant, by the Nullstellensatz there exists $\lambda\in\Q$ such that $hQ= \lambda Q$, i.e. 
$\mathcal H^t\mathcal Q \mathcal H=\lambda \mathcal Q$. By taking determinants, we conclude that $\lambda^3 = 1$, hence $\lambda = 1$. Putting the  explicit matrices into the formula, we obtain 
\begin{equation}\label{equation}
a_h = c_h\quad\text{and}\quad d_h  = \frac{a_h (a_h - 1)}{2}.
\end{equation} 


Since $w\in N^1(X)$, there is a primitive element $\overline w\in N^1(X)$ and a positive integer $p$ such that $w=p\overline w$. We have $a_hp\overline w=a_h w=hw_1 - w_1\in N^1(X)$, hence the number $a_hp$ must be an integer. 
Consider the group homomorphism
$$ \tau\colon \mcal A(X) \to \Z, \qquad h\mapsto pa_h.$$
By (\ref{equation}), $\tau $ is injective, and therefore $\mathcal A(X)\simeq\Z$. Thus $\mathcal A(X)$ is abelian of rank $1$. 
%
\end{proof}

\bibliographystyle{amsalpha}

\bibliography{biblio}
\end{document}